\documentclass[11pt,a4paper]{article}

\usepackage[ngerman,english]{babel}
\usepackage[utf8]{inputenc}

\usepackage[DIV=12]{typearea}

\usepackage[affil-it]{authblk}

\usepackage{xparse}

\usepackage{amsmath, amsfonts, amssymb, amsthm}
\usepackage{mathrsfs}
\usepackage{mathtools}
\usepackage{enumerate}

\usepackage[noadjust]{cite}

\newcommand{\R}{\mathbb{R}}

\newcommand{\N}{\mathbb{N}}

\newcommand{\ee}{\mathrm{e}}

\DeclareDocumentCommand\dd{ o g d() }{
	\IfNoValueTF{#2}{
		\IfNoValueTF{#3}
			{\mathrm{d}\IfNoValueTF{#1}{}{^{#1}}}
			{\mathinner{\mathrm{d}\IfNoValueTF{#1}{}{^{#1}}\argopen(#3\argclose)}}
		}
		{\mathinner{\mathrm{d}\IfNoValueTF{#1}{}{^{#1}}#2} \IfNoValueTF{#3}{}{(#3)}}
	}

\newcommand{\del}{\partial}

\newcommand{\W}{\mathscr{W}}

\DeclareMathOperator{\dist}{dist}

\DeclareMathOperator{\id}{Id}

\newcommand{\vcc}{\vcentcolon}

\DeclarePairedDelimiter\abs{\lvert}{\rvert}
\DeclarePairedDelimiter\norm{\Vert}{\rVert}

\theoremstyle{plain}
\newtheorem{theorem}{Theorem}[section]
\newtheorem{lemma}[theorem]{Lemma}
\newtheorem{proposition}[theorem]{Proposition}

\theoremstyle{definition}

\theoremstyle{remark}
\newtheorem{remark}[theorem]{Remark}
\newtheorem{example}[theorem]{Example}

\title{Mean field limit for interacting systems on co-evolving networks}
\author{Sebastian Throm \thanks{\texttt{sebastian.throm@umu.se}}}
  
 \affil{\em Ume{\aa} University,  Department of Mathematics and Mathematical Statistics, 901 87 Umeå }
\date{}

\begin{document}

\maketitle

\begin{abstract}
Interacting particle systems are in frequent use to model collective behaviour in various situations and applications. For many systems, the interaction between the agents is restricted to an underlying network structure and often, the latter also evolves in time with its dynamics coupled to the evolution of the particles. Due to their relevance for applications such systems on adaptive or co-evolutionary networks have received increasing interest in recent years. In particular, a fundamental question concerns the behaviour of the system in the infinite particle limit. In this work we provide a mean-field description for a general particle system which exhibits non-locality in time (memory). The result applies particularly to a large class of systems on co-evolving networks including non-linear weight dynamics.
\end{abstract}

\section{Introduction}

Collective behaviour of large multi-agent systems can be observed in many real-world situations and consequently, there is a strong interest in deriving and analysing corresponding mathematical models \cite{BCK20,CKJ02,DVN24,NPT10,WHK22}. A common modelling approach consists in considering the agents in the system as point masses (particles) and neglect their spatial dimensions. In situations where the number of particles is large, one is often not interested in the precise evolution of a single particle but rather the averaged behaviour of the system when the number of particles tends to infinity. This leads to the derivation of continuum descriptions for the a priori discrete systems and under the assumption of all-to-all coupling between the particles there is a well established theory available on the derivation of \emph{mean-field equations} \cite{BrH77,NeW74,Dob79a,Gol16}. More precisely, if $z_k(t)\in\R^d$ describes the state of the $k$'th particle at time $t$ the \emph{empirical measure} is given by
\begin{equation*}
 \mu^N_t=\sum_{k=1}^{N}\delta_{z_k(t)}
\end{equation*}
where $N$ is the total number of particles and $\delta$ denotes the Dirac measure. Under suitable assumption of the (pairwise) interaction of the particles one is interested in the limit $\mu_t=\lim_{N\to\infty}\mu^N_t$ of the empirical measure.

Many systems in real world situations lack an all-to-all coupling but rather interact according to an underlying network structure where the particles are located at the nodes of a network/graph and interaction between particles only takes place if they are directly linked via an edge. In recent years there has been a strong interest in studying such systems and particularly to generalise the concept of deriving their mean-field limit \cite{BBV08,PoG16,Lu20,CDG20,MeM21,WHK22,Bur22,KHT22,DDZ22,BCW23,KuT19,MeM21}. A fundamental ingredient is a continuous description of the underlying network in the limit $N\to\infty$. Based on the concept of \emph{graphons} and \emph{graph convergence}, there has been pioneering work mimicking the classical derivation of mean-field equations \cite{ChM16,KaM18}. More recently, this has been further developed to include more general graph structures \cite{GkK22,JPS25,KuX22a}. These approaches are however limited to \emph{static networks} which are fixed for all times.

On the other hand, many important applications exhibit networks which change dynamically. Often, this dynamics is also coupled to the evolution of the particles and the networks are denoted \emph{adaptive} or \emph{co-evolving} \cite{BGK23a,GDB06,NGW23a,Bur22,SBL23,RuP23,DBB23,JuM23,GKX25}. Such systems naturally exhibit more complex dynamics which makes their analysis rather challenging. In particular, the co-evolving structure induces non-local/memory effects (see below).

\subsection{The model}

We will derive a mean-field description for interacting particle systems with a general underlying co-evolving network structure, i.e.\@ which evolves in time with its dynamics coupled to the particle system itself. More precisely, we consider a system of $N$ particles. The state of the $k$th particle at time $t$ is described by $\phi^k_t$. Moreover, the coupling of the particles is described by the adjacency matrix $(w_{kl})\in \R^{N\times N}$ of the underlying network whose entries describe the coupling strength between the $k$th and $\ell$th particle. We assume in particular, that $w_{k\ell}$ evolves in time as well and the dynamics is driven by the current state and the evolution of the states $\phi^k_t$ and $\phi^\ell_t$. This leads to the following coupled ODE system 

\begin{equation}\label{eq:disc:model:1}
 \begin{split}
  \dot{\phi}^k_t&=\frac{1}{N}\sum_{\ell=1}^{N}w_{k\ell}C(\phi^k_t,\phi^\ell_t)\\
  \dot{w}_{k\ell}&=F(w_{k\ell},\phi^k_t,\phi^\ell_t).
 \end{split}
\end{equation}
Here $C\colon \R^d\times \R^d \to \R$ as well as $F\colon \R^{N\times N}\times \R^d\times \R^d\to \R$ are sufficiently regular functions while the system is complemented with suitable initial data.

Our primary goal consists in deriving a generalisation of the \emph{mean-field limit}, i.e.\@ a continuous limit approximation of the particle dynamics as $N\to\infty$. More precisely, the basic idea is to look at the sequence of empirical measures
\begin{equation}\label{eq:emp:measure:1}
 \mu^N_t=\frac{1}{N}\sum_{k=1}^{N}\delta_{\phi_t^k}
\end{equation}
and study the corresponding limit $N\to\infty$. The main challenge consists in the limiting behaviour of the underlying network structure. More precisely, we have to deal on the one hand with the fact that the weights $w_{k\ell}$ evolve depending on the system dynamics. Moreover, we require a continuous approximation of the a priori discrete network structure (and particularly its initial state) as $N\to \infty$. For the second problem, we will rely on the theory of \emph{graphons} and \emph{graph convergence} which have been frequently employed in similar contexts for static networks in recent works and we follow mainly \cite{ChM16,KaM18}. In particular, we restrict ourselves to a relatively simple situation compared to more advanced notions of graph limits here to keep the technical problems on this aspect to a minimal level. In fact our main focus is on the first challenge concerning the adaptivity. To deal with this, we note that, assuming sufficient regularity for $\phi^k$ and $\phi^\ell$, we consider the equation 
\begin{equation}\label{eq:network:dyn:1}
 \dot{w}_{k\ell}=F(w_{k\ell},\phi^k_t,\phi^\ell_t); \qquad w_{k\ell}(t_0)=w_{k\ell}^{\text{in}}
\end{equation}
with $\phi^k$ and $\phi^\ell$ as (non-local) parameters. If $F$ is regular enough, we can solve this equation which provides us with a flow map 
\begin{equation}\label{eq:network:flow}
 \Phi_{t_0,t}[w_{k\ell}^{\text{in}},\phi^k|_{[t_0,t]},\phi^{\ell}|_{[t_0,t]}].
\end{equation}
More precisely, we have the following statement.
\begin{lemma}\label{Lem:weight:dyn}
 Let $F\colon[0,T]\times \R\times \R^d\times \R^d\to\R$ be continuous with respect to the first and globally Lipschitz continuous with respect to the remaining components with Lipschitz constant $L_F$, i.e.\@ 
 \begin{equation*}
  \abs{F_t(w,a,b)-F_t(v,\tilde{a},\tilde{b}}\leq L_F(\abs{w-v}+\abs{a-\tilde{a}}+\abs{b-\tilde{b}}).
 \end{equation*}
 Then for any given $\gamma,\tilde{\gamma}\in C([0,T],\R^d)$ the equation
 \begin{equation*}
  \dot{w}_t=F_t(w_t,\gamma_t,\tilde{\gamma}_t) \qquad w(0)=w_0
 \end{equation*}
has a unique solution $w[\gamma,\tilde{\gamma}]$. In particular, $w$ is Lipschitz continuous with respect to $\gamma$ and $\tilde{\gamma}$ with Lipschitz constant $\ee^{L_FT}$.
\end{lemma}
\begin{proof}
 We sketch the proof which follows by a straightforward application of Banach's fixed point theorem on the space $C([0,T],\R^d)$ equipped with the norm $\norm{w}\vcc=\sup_{t\in[0,T]}\ee^{-\alpha t}\abs{w(t)}$ with a constant $\alpha>L_F$. In fact the problem can be reformulated as the fixed-point problem
 \begin{equation*}
  w(t)=w_0+\int_{0}^{t}F_{s}(w,\gamma_s,\tilde{\gamma}_s)\dd{s}=\vcc \mathcal{F}[w]
 \end{equation*}
and one immediately checks that $\mathcal{F}\colon C([0,T],\R^d)\to C([0,T],\R^d)$ is a contraction since $\alpha>L_F$. This provides a unique solution $w[\gamma,\tilde{\gamma}]\in C([0,T],\R^d)$. The Lipschitz-continuity with respect to $\gamma$ and $\tilde{\gamma}$ follows from the corresponding Lipschitz continuity of $F$ and Gronwall's inequality. In fact, we have
\begin{multline*}
 \abs{w_t[\gamma,\tilde{\gamma}]-w_t[\rho,\tilde{\rho}]}=\abs[\bigg]{\int_{0}^{t}F_s(w_s[\gamma,\tilde{\gamma}],\gamma_s,\tilde{\gamma}_s)-F_s(w_s[\rho,\tilde{\rho}],\rho_s,\tilde{\rho}_s)\dd{s}}\\*
 \leq L_F\int_{0}^{t}\abs{w_s[\gamma,\tilde{\gamma}]-w_s[\rho,\tilde{\rho}]}+\abs{\gamma_s-\rho_s}+\abs{\tilde{\gamma}_s-\tilde{\rho}_s}\dd{s}\\*
 \leq L_F\int_{0}^{t}\abs[\big]{w_s[\gamma,\tilde{\gamma}]-w_s[\rho,\tilde{\rho}]}\dd{s}+L_F t\bigl(\norm{\gamma-\rho}_{C([0,T],\R^d)}+\norm{\tilde{\gamma}-\tilde{\rho}}_{C([0,T],\R^d)}\bigr)
\end{multline*}
from which we conclude by means of Gronwall's inequality.

\end{proof}

\begin{remark}
 For $t_0=0$ by abuse of notation we write 
 \begin{equation*}
  \Phi_{0,t}[w_{k\ell}^{\text{in}},\phi^k|_{[t_0,t]},\phi^{\ell}|_{[t_0,t]}]=\Phi_{t}[w_{k\ell}^{\text{in}},\phi^k,\phi^{\ell}].
 \end{equation*}
\end{remark}
Using \eqref{eq:network:flow} we can replace the weights $w_{k\ell}$ in \eqref{eq:disc:model:1} which leads to 
\begin{equation}\label{eq:disc:model:2}
 \dot{\phi}^k_t=\frac{1}{N}\sum_{\ell=1}^{N}\Phi_{t}[w_{k\ell}^{\text{in}},\phi^k,\phi^{\ell}]C(\phi^k_t,\phi^\ell_t).
\end{equation}
The advantage of this equation over \eqref{eq:disc:model:1} is that it is a closed system for $\phi^k$ while the underlying network structure only enters explicitly through the initial datum $w_{k\ell}^{\text{in}}$ which can be treated essentially using the same techniques as for static networks. However, the network structure and its evolution is of course encoded in $\Phi_t$. In particular, the price we have to pay for this decoupling is that we are now concerned with a non-local ODE system since $\Phi_t$ in general depends on the whole trajectory of its input parameters.

These considerations motivate to consider the following non-local in time system of equations
\begin{equation}\label{eq:disc:model:3}
 \dot{\phi}^k_t=\frac{1}{N}\sum_{\ell=1}^{N}K_{t}(w_{k\ell}^{\text{in}},\phi^k|_{[0,t]},\phi^{\ell}|_{[0,t]});\qquad \phi^{k}_0=\phi^{k,\text{in}}
\end{equation}
where now $K_t\colon \R\times C([0,t],\R^d)\times C([0,t],\R^d) \to \R$ is a functional on the space of continuous functions.

An examples of a system which fits into this framework is given as follows:

\begin{example}
 This example is a specific case of \eqref{eq:disc:model:2} and motivated by adaptively coupled phase oscillators which arise e.g.\@ in neuroscience \cite{HNP16}. Precisely, one considers systems of the form
 \begin{equation}\label{eq:example:2}
  \begin{split}
  \dot{\phi}^k_t&=\frac{1}{N}\sum_{\ell=1}^{N}w_{k\ell}C(\phi^k_t,\phi^\ell_t)\\
  \dot{w}_{k\ell}&=-aw_{k\ell}+g(\phi^k,\phi^\ell).
  \end{split}
 \end{equation}
 The second equation can be integrated explicitly by means of Duhamel's formula which yields
 \begin{equation*}
  w_{k\ell}(t)=w_{k\ell}\ee^{-at}+ \int_{0}^{t}g(\phi^k_s,\phi_s^\ell)\ee^{-a(t-s)}\dd{s}.
 \end{equation*}
 Plugging this expression into the first equation in \eqref{eq:example:2} provides a closed (non-local in time) system for $\phi^k$, i.e.\@ 
 \begin{equation*}
  \dot{\phi}^k_t=\frac{1}{N}\sum_{\ell=1}^{N}\biggl(w_{k\ell}\ee^{-at}+ \int_{0}^{t}g(\phi^k_s,\phi_s^\ell)\ee^{-a(t-s)}\dd{s}\biggr)C(\phi^k_t,\phi^\ell_t).
 \end{equation*}
\end{example}

\subsection{Notation and assumptions}

We collect in the following some notation and the main assumptions which we will use throughout this work. 

For $T>0$ fixed let $C([0,T],\R^d)$ be the space of $\R^d$-valued continuous functions on the interval $[0,T]$ equipped with the supremum norm $\norm{f}_{C([0,T],\R^d)}=\sup_{t\in[0,T]}\abs{f(t)}$ for which $C([0,T],\R^d)$ is a Banach space. Moreover, to simplify notation at some places we will also write $C([0,T])$.

We denote by $e_t\colon C([0,T])\to \R^d$ the evaluation map at $t\in [0,T]$, i.e.\@ for $f\in C([0,T],\R^d)$ we have $e_t f=f(t)$. Moreover, $r_t\colon C([0,T],\R^d)\to C([0,t],\R^d)$ denotes the restriction, i.e.\@ $r_t f=f|_{[0,t]}$.

To describe the underlying graph structure, we will mainly follow the approach in \cite{ChM16,KaM18} which builds on the theory of graph limits \cite{LoS06,Lov12,BCL08}. More precisely, we parametrise graphs over a fixed reference domain $I=[0,1)$ by considering for $N\in\N$ the partition $I=\cup_{k=1}^{N}I_k^N$ with $I_k^N=[(k-1)/N,k/N)$. To simplify the notation at some places, we might also drop the superscript $N$. We also denote $[N]=\{1,\ldots,N\}$.

A (signed) \emph{graphon} will be a function $\W\colon I\times I\to \R$. Given a graphon $\W$ one can construct associated sequences of finite graphs converging to this limiting graphon: e.g.\@ if we denote $x_k=(k-1/2)/N$ the center of $I_k$ one defines $w^N_{k\ell}=\W(x_k,x_\ell)$. Alternatively, we might take the average
\begin{equation}\label{eq:average:weight}
 w_{k\ell}^N=N^2\int_{I_k^N\times I_\ell^N}\W(x,y)\dd{x}\dd{y}.
\end{equation}
 Conversely, given a finite graph encoded via the adjacency matrix $w_{k\ell}$ we can build a corresponding (piecewise constant) graphon 
\begin{equation}\label{eq:step:graphon}
  \W(x,y)=\sum_{k,\ell=1}^{N}w_{k\ell}\chi_{I_k}(x)\chi_{I_\ell}(y)
\end{equation}
where $\chi_A$ denotes the characteristic function of the set $A$.

If $(X,d)$ is a complete, separable, normed space we denote by $\mathcal{P}(X)$ the space of corresponding Borel probability measures on $X$, while $\mathcal{P}_q(X)$ denotes the subset of measures with bounded $q$th moment, i.e.\@
\begin{equation*}
 \int_{\R^d}\bigl(d(x_0,x)\bigr)^q\dd{\mu(x)}<\infty \qquad \text{for any } x_0\in X.
\end{equation*}
Moreover, on $\mathcal{P}(X)$ we consider the Wasserstein-1 distance denoted by $\dist$ which is defined by
\begin{equation*}
 \dist(\mu,\nu)=\inf_{\pi\in\Pi(\mu,\nu)}\int_{X\times X}d(x,y)\dd{\pi}(x,y)
\end{equation*}
where $\Pi(\mu,\nu)$ is the set of couplings between $\mu$ and $\nu$, i.e.\@ probability measures on $X\times X$ with first and second marginal $\mu$ and $\nu$ respectively. Equivalently, $\dist(\mu,\nu)$ is also given by
\begin{equation*}
 \dist(\mu,\nu)=\sup_{\norm{\varphi}_{\text{Lip}}\leq 1}\abs[\bigg]{\int_{X}\varphi\dd{\mu}-\int_{X}\varphi\dd{\nu}}
\end{equation*}
where the supremum is taken over all Lipschitz functions $\varphi\colon X\to\R $ with Lipschitz constant less than $1$ (see e.g.\@ \cite{Vil09}).

Following \cite{KaM18} we consider further the set of measurable $\mathcal{P}(X)$-valued functions $\mathcal{P}^{I}(X)=\{\mu\colon I\ni x\mapsto \mu^{x}\in \mathcal{P}(X) \text{ measurable}\}$ which we equip with the associated distance
\begin{equation}\label{eq:WI}
\begin{split}
 \dist^{I}(\mu,\nu)=\int_{I}\dist(\mu^x,\nu^x)\dd{x}&=\int_{I}\inf_{\pi\in\Pi(\mu^x,\nu^x)}\int_{X\times X}\abs{\zeta-\tilde{\zeta}}\dd{\pi}(\zeta,\tilde{\zeta})\dd{x}\\
 &=\int_{I}\sup_{\norm{\varphi}_{\text{Lip}}\leq 1}\abs[\bigg]{\int_{X}\varphi\dd{\mu^x}-\int_{X}\varphi\dd{\nu^x}}\dd{x}.
 \end{split}
\end{equation}
Given two measure spaces $(X,\sigma_X)$, $(Y,\sigma_Y)$ and a measurable map $f\colon X\to Y$ we denote as usual the \emph{push-forward} of a measure $\mu$ on $X$ under $f$ by $f\#\mu$ which is defined through the formula
\begin{equation*}
 f\#\mu(B)=\mu(f^{-1}(B))\quad \text{for all } B\in \sigma_Y
\end{equation*}
where $f^{-1}(B)$ is the pre-image of $B$ under $f$. We also note that 
\begin{equation*}
 \int_{Y}\varphi(y)\dd{f\#\mu}(y)=\int_{X}\varphi\circ f(x)\dd{\mu}(x).
\end{equation*}
Let $K_t\colon \R\times C([0,t],\R^d)\times C([0,t],\R^d) \to \R$ satisfy the following regularity assumptions:
\begin{equation}\label{Ass:K1}
 K_t \text{ is bounded, i.e.\@ } \sup_{t\in[0,T]}\abs{K_t(a,r_t f,r_t g)}\leq C \text{ for all } a\in\R \text{ and all } f,g\in C([0,T],\R^d).
\end{equation}
\begin{equation}\label{Ass:K2}
 t\mapsto K_t \text{ is continuous on } [0,T]
\end{equation}
\begin{equation}\label{Ass:K3}
\begin{gathered}
  K_t \text{ is Lipschitz continuous with constant } L_K \text{ uniformly with respect to }t\in[0,T],\text{ i.e.\@ } \\*
  \abs{K_t(a,r_t f^1,r_t g^1)-K_t(b,r_t f^2,r_t g^2)}\leq L_{K}(\abs{a-b}+\norm{f^1-g^1}_{C([0,t])}+\norm{f^2-g^2}_{C([0,t])}).
\end{gathered} 
\end{equation}
Assumption \eqref{Ass:K3} implies in particular
\begin{equation*}
 \sup_{t\in[0,T]}\abs{K_t(a,r_t f^1,r_t g^1)-K_t(b,r_t f^2,r_t g^2)}\leq L_{K}(\abs{a-b}+\norm{f^1-g^1}_{C([0,T])}+\norm{f^2-g^2}_{C([0,T])})
\end{equation*}
\begin{remark}
 In view of \eqref{Ass:K2}, $K_\cdot$ can be regarded as a map to $C([0,T],\R^d)$.
\end{remark}

\subsection{Main results}
The aim of this work consists in establishing a mean-field description for systems of the form \eqref{eq:disc:model:3} which will be given by a measure on the path space $C([0,T],\R^d)$ which reflects the non-local in time structure. We also note that we do not strive for the most general result concerning our assumptions but to present the results in a relatively easy setting to avoid technical difficulties. 

We first note the following statement which provides existence of a unique solution to discrete systems of the form \eqref{eq:disc:model:3}.

\begin{theorem}\label{Thm:disc:well:posed}
 Let $K$ satisfy \eqref{Ass:K1}--\eqref{Ass:K3}. The system \eqref{eq:disc:model:3} has a unique solution for every given $(\phi^{k,\text{in}})_k\in (\R^d)^N$
\end{theorem}
This statement can be proved by a straightforward application of Banach's fixed point theorem similarly to the proof of Lemma~\ref{Lem:weight:dyn} and we omit the details.

The first main result of this work concerns the derivation of the mean-field limit for \eqref{eq:disc:model:3} with a sequence of initial networks generated from a Lipschitz graphon. More precisely, if $\phi^{N,k}$ denotes the solution to \eqref{eq:disc:model:3} we will show that the corresponding empirical measure $\mu^{N}$ has a limit $\mu$ as $N\to\infty$. The precise formulation of the result can be found in Theorem~\ref{Thm:mean:field:1}. 

Our second main result concerns the relaxation of the Lipschitz condition on the graphon. In fact, following the approach in \cite{KaM18}, we will derive a mean-field description of \eqref{eq:disc:model:3} under the weaker continuity assumption \eqref{eq:W:cont}. The precise result can be found in Theorem~\ref{Thm:mean:field:2}.¸

The remainder of the article is organised as follows. In Section~\ref{Sec:Lip:graph} we demonstrate the derivation of the mean-field limit for \eqref{eq:disc:model:3} for a sequence of Lipschitz graphs. More precisely, we combine the classical methods in \cite{Gol16} with the approach in \cite{ChM16} and adapt them in a suitable way to the non-local in time system \eqref{eq:disc:model:3}. The first step consists in constructing a suitable system of characteristic equations associated to the discrete model. The key step is Dobrushin's estimate which is given in Proposition~\ref{Prop:Dobrushin} and which gives a bound on the difference of the mean field limit and the empirical measure at time $t$ in terms of the corresponding initial distance over a finite time horizon. We also briefly discuss the relation between the mean-field limit and the corresponding continuum limit of \eqref{eq:disc:model:3} in analogy to \cite{AyP21}.

In Section~\ref{Sec:non:Lip:graph} we outline how the results from Section~\ref{Sec:Lip:graph} can be generalised to non-Lipschitz graph sequences by adapting the approach in \cite{KaM18}.

\section{The mean field limit for Lipschitz graphs}\label{Sec:Lip:graph}

In this section, we derive the mean-field limit for \eqref{eq:disc:model:3} assuming that the underlying graph sequence is generated from a Lipschitz continuous graphon. More precisely, following \cite{ChM16}, for $N\in\N$, we fix a discretisation $X_N=\{x^{N,1},\ldots,x^{N,N}\}$ of $I=[0,1]$ such that 
\begin{equation}\label{eq:disc}
 \lim_{N\to\infty}\frac{1}{N}\sum_{k=1}^{N}f(x^{N,k})=\int_{I}f(x)\dd{x} \qquad \text{for all } f\in C(I).
\end{equation}
A canonical choice would e.g.\@ be $x^{N,k}=\frac{k-1/2}{N}$ for $k=1,\ldots,N$, i.e.\@ the center points of $I_k$.
Let $\W\colon I\times I\to \R$ be symmetric and Lipschitz continuous, i.e.\@
\begin{equation}\label{eq:ass:Lip:W}
 \abs{\W(x,y)-\W(\tilde{x},\tilde{y}}\leq L_\W\bigl(\abs{x-\tilde{x}}^2+\abs{y-\tilde{y}}^2\bigr)^{1/2}\quad \text{and}\quad \W(x,y)=\W(y,x) \quad \text{for all } x,\tilde{x},y,\tilde{y}\in I.
\end{equation}
We consider the graph sequence $G^N=(V(G^N),E(G^N))$ generated by $\W$ and the discretisation $X_N$, i.e.\@ we identify the vertices $V(G^N)=[N]$ and the edges $E(G^N)=\{(k,\ell)\in I\times I \;\vert\; \W(x^{N,k},x^{N,\ell})\neq 0 \}$. To each edge $(k,\ell)\in E(G^N)$ we associate the weight 
\begin{equation}\label{eq:Lip:weight}
w_{k\ell}=\W(x^{N,k},x^{N,\ell}). 
\end{equation}
With this, we can consider the augmented system 
\begin{equation}\label{eq:disc:model:4}
 \begin{split}
   \dot{\phi}^k_t&=\frac{1}{N}\sum_{\ell=1}^{N}K_{t}(w_{k\ell}^{\text{in}},\phi^k|_{[0,t]},\phi^{\ell}|_{[0,t]});\qquad \phi^{k}_0=\phi^{k,\text{in}}\\
   \dot{x}^{N,k}&=0.
 \end{split}
\end{equation}
The corresponding empirical measure on $C([0,T],\R^d)\times I$ is then defined as
\begin{equation}\label{eq:emp:measure:2}
 \mu^N=\frac{1}{N}\sum_{k=1}^{N}\delta_{(\phi_\cdot^k,x^{N,k})}.
\end{equation}
We note moreover, that the evaluation map $e_t$ induces a measure on $\R^d\times I$ via 
\begin{equation}\label{eq:emp:measure:1:t}
 \mu^N_t\vcc=(e_t,\id)\#\mu^N=\frac{1}{N}\sum_{k=1}^{N}\delta_{(\phi_t^k,x^{k,N})}.
\end{equation}

\subsection{The characteristic equation}

We adapt mainly the classical approach and notation as outlined in \cite{Gol16} to derive the mean-field limit for interacting systems in our situation. More precisely, we consider first the corresponding characteristic equations
\begin{equation}\label{eq:charact:1}
 \begin{split}
  \del_{t}Z^x_t(\zeta^{\text{in}},\mu_0)&=\int_{C([0,t])\times I}K_t(\W(x,y),r_t Z^x_\cdot(\zeta^{\text{in}},\mu_0),r_t\gamma_\cdot)\dd{\mu(\gamma,y)}\\
  \mu&=(Z^{\cdot}_{\cdot}(\cdot,\mu_0),\id)\#\mu_0\\
  Z^x_0(\zeta^{\text{in}},\mu_0)&=\zeta^{\text{in}}
 \end{split}
\end{equation}
where $\zeta^{\text{in}}\in\R^d$, $\mu_0\in \mathcal{P}(\R^d\times I)$. We also note that we consider $Z^\cdot_{\cdot}(\cdot,\mu_0)\colon \R^d\times I\to C([0,T],\R^d)$ via $\R^d\times I\ni (\zeta,x)\mapsto Z^x_{\cdot}(\zeta,\mu_0)\in C([0,T],\R^d)$.
We can equivalently reformulate \eqref{eq:charact:1} as
\begin{equation}\label{eq:charact:2}
 \begin{split}
  \del_{t}Z^x_t(\zeta^{\text{in}},\mu_0)&=\int_{\R^d\times I}K_t(\W(x,y),r_t Z^x_\cdot(\zeta^{\text{in}},\mu_0),r_t Z^y_{\cdot}(\zeta,\mu_0))\dd{\mu_0(\zeta,y)}\\*
  Z^x_0(\zeta^{\text{in}},\mu_0)&=\zeta^{\text{in}}.
 \end{split}
\end{equation}
The following proposition provides the existence of a unique solution to the characteristic equation.
\begin{proposition}\label{Prop:ex:charact:flow}
 Assume that $K$ satisfies \eqref{Ass:K1}--\eqref{Ass:K3} and $\W$ satisfies \eqref{eq:ass:Lip:W}. For each $\mu_0\in\mathcal{P}_1(\R^d\times I)$ the equation \eqref{eq:charact:2} has a unique solution $Z\in C([0,T]\times I\times \R^d,\R^d)$ with $[0,T]\times I\times \R^d \ni (t,x,\zeta^{\text{in}})\mapsto Z_{t}^{x}(\zeta^{\text{in}},\mu_0)\in\R^d$. Moreover, $Z$ is differentiable with respect to the time variable $t$.
\end{proposition}
The proof follows by classical arguments in the line of the proof of Banach's fixed-point theorem. In particular, the statement can be shown by adapting the proof of \cite[Theorem 1.3.2]{Gol16} in a straightforward manner using the norm $\norm{Z}=\sup_{\zeta\in \R^d}\sup_{(t,x)\in [0,T]\times I}\frac{\abs{Z_{t}^{x}(\zeta)}}{1+\abs{\zeta}}$.

The characteristics satisfy the following Lipschitz condition.
\begin{lemma}\label{Lem:Lip:flow}
 For $K$ satisfying \eqref{Ass:K1}--\eqref{Ass:K3} and $\W$ satisfying \eqref{eq:ass:Lip:W}, the characteristics $Z^x_t(\zeta^{\text{in}},\mu_0)$ are Lipschitz continuous with respect to $x$ and the initial datum $\zeta^{\text{in}}$. More precisely, 
\begin{equation*}
 \begin{split}
    \norm{Z^x_\cdot(\zeta^{\text{in}},\mu_0)-Z^{\tilde{x}}_\cdot(\tilde{\zeta}^{\text{in}},\mu_0)}_{C([0,t])}&\leq \ee^{L_K t}\abs{\zeta^{\text{in}}-\tilde{\zeta^{\text{in}}}}+L_\W(\ee^{L_K t}-1)\abs{x-\tilde{x}}\\
    &\leq \sqrt{2}\max\{1,L_\W\}\ee^{L_K t}\bigl(\abs{\zeta^{\text{in}}-\tilde{\zeta}^{\text{in}}}^2+\abs{x-\tilde{x}}^2\bigr)^{1/2}
 \end{split}
 \end{equation*}
  uniformly in $\mu_0$. In particular, we have
 \begin{equation*}
  \begin{split}
     \norm{Z^x_\cdot(\zeta^{\text{in}},\mu_0)-Z^{\tilde{x}}_\cdot(\tilde{\zeta}^{\text{in}},\mu_0)}_{C([0,T])}&\leq \ee^{L_K T}\abs{\zeta^{\text{in}}-\tilde{\zeta}^{\text{in}}}+L_\W(\ee^{L_K T}-1)\abs{x-\tilde{x}}\\
  &\leq \sqrt{2}\max\{1,L_\W\}\ee^{L_K T}\bigl(\abs{\zeta^{\text{in}}-\tilde{\zeta}^{\text{in}}}^2+\abs{x-\tilde{x}}^2\bigr)^{1/2}.
  \end{split}
 \end{equation*}
\end{lemma}

\begin{proof}
Since $\mu_0$ is fixed we drop the dependence of $Z$ on this parameters during the proof to simplify the notation. From \eqref{eq:charact:2} together with \eqref{Ass:K3} and \eqref{eq:ass:Lip:W} we deduce
 \begin{multline*}
 Z_t^{x}(\zeta^{\text{in}})-Z_t^{\tilde{x}}(\tilde{\zeta}^{\text{in}})=\zeta^{\text{in}}-\tilde{\zeta}^{\text{in}}\\*
+\int_{0}^{t}\int_{\R\times I}\Bigl(K_s(\W(x,y),r_s Z^x_\cdot(\zeta^{\text{in}}),r_s Z^y_{\cdot}(\zeta))-K_s(\W(\tilde{x},y),r_s Z^{\tilde{x}}_\cdot(\tilde{\zeta}^{\text{in}}),r_s Z^y_{\cdot}(\zeta))\Bigl)\dd{\mu_0(\zeta,y)}\dd{s}\\*
\leq \zeta^{\text{in}}-\tilde{\zeta}^{\text{in}}+L_K\int_{0}^{t}\int_{\R\times I}\abs{\W(x,y)-\W(\tilde{x},y)}+\norm{Z^x_\cdot(\zeta^{\text{in}})-Z^{\tilde{x}}_\cdot(\tilde{\zeta}^{\text{in}})}_{C([0,s])}\dd{\mu_0(\zeta,y)}\dd{s}\\*
\leq \zeta^{\text{in}}-\tilde{\zeta}^{\text{in}}+L_KL_\W t \abs{x-\tilde{x}}+L_K\int_{0}^{t}\norm{Z^x_\cdot(\zeta^{\text{in}})-Z^{\tilde{x}}_\cdot(\tilde{\zeta}^{\text{in}})}_{C([0,s])}\dd{s}.
\end{multline*}
From this estimate the claim easily follows, taking the supremum in $t$ of the absolute value and applying Gronwall's inequality.
\end{proof}

The next proposition states that the characteristics are Lipschitz continuous in the initial measure with respect to the Wasserstein distance which is the key ingredient in the proof of Dobrushin's estimate.
\begin{proposition}\label{Prop:cont:dep:flow}
Assume \eqref{Ass:K1}--\eqref{Ass:K3} and let $\W$ satisfy \eqref{eq:ass:Lip:W}. The characteristics¸ provided by Proposition~\ref{Prop:ex:charact:flow} satisfies the following Lipschitz condition with respect to the initial measure
\begin{multline*}
 \int_{\R\times I}\norm{Z_\cdot^{x}(\zeta^{\text{in}},\mu_0)-Z_\cdot^{x}(\zeta^{\text{in}},\tilde{\mu}_0)}_{C([0,t])}\dd{\tilde{\mu}_0(\zeta^{\text{in}},x)}\\*
 \leq \sqrt{2}\max\{1,L_\W\} \dist(\mu_0,\tilde{\mu}_0) (2\ee^{2L_K t}-\ee^{L_K t}).
\end{multline*}
\end{proposition}

\begin{proof}
 From the definition of $Z$ via \eqref{eq:charact:2} we have
\begin{multline*}
 Z_t^{x}(\zeta^{\text{in}},\mu_0^x)-Z_t^{x}(\zeta^{\text{in}},\tilde{\mu}_0)\\*
\shoveleft{=\int_{0}^{t}\int_{\R^d\times I}K_s(\W(x,y),r_s Z^x_\cdot(\zeta^{\text{in}},\mu_0),r_s Z^y_{\cdot}(\zeta,\mu_0))\dd{\mu_0(\zeta,y)}\dd{s}}\\*
\shoveright{-\int_{0}^{t}\int_{\R^d\times I}K_s(\W(x,y),r_s Z^x_\cdot(\zeta^{\text{in}},\tilde{\mu}_0),r_s Z^y_{\cdot}(\zeta,\tilde{\mu}_0))\dd{\tilde{\mu}_0(\zeta,y)}\dd{s}}\\*
\shoveleft{=\int_{0}^{t}\int_{\R^d\times I}K_s(\W(x,y),r_s Z^x_\cdot(\zeta^{\text{in}},\mu_0),r_s Z^y_{\cdot}(\zeta,\mu_0))\dd{\mu_0(\zeta,y)}}\\*
\shoveright{-\int_{\R^d\times I}K_s(\W(x,y),r_s Z^x_\cdot(\zeta^{\text{in}},\mu_0),r_s Z^y_{\cdot}(\zeta,\mu_0))\dd{\tilde{\mu}_0(\zeta,y)}\dd{s}}\\*
\shoveleft{+\int_{0}^{t}\int_{\R^d\times I}\Bigl(K_s(\W(x,y),r_s Z^x_\cdot(\zeta^{\text{in}},\mu_0),r_s Z^y_{\cdot}(\zeta,\mu_0))}\\*
-K_s(\W(x,y),r_s Z^x_\cdot(\zeta^{\text{in}},\tilde{\mu}_0),r_s Z^y_{\cdot}(\zeta,\tilde{\mu}_0))\Bigl)\dd{\tilde{\mu}_0(\zeta,y)}\dd{s}.
\end{multline*}
Together with \eqref{Ass:K3} and Lemma~\ref{Lem:Lip:flow} this yields
\begin{multline*}
 Z_t^{x}(\zeta^{\text{in}},\mu_0^x)-Z_t^{x}(\zeta^{\text{in}},\tilde{\mu}_0)\\*
 \shoveleft{\leq \sqrt{2}\max\{1,L_\W\}\ee^{L_K t}\dist(\mu_0,\tilde{\mu}_0)}\\*
 +L_K\int_{0}^{t}\int_{\R^d\times I}\Bigl(\norm{Z^x_\cdot(\zeta^{\text{in}},\mu_0)-Z^x_\cdot(\zeta^{\text{in}},\tilde{\mu}_0)}_{C([0,s])}\\*
+\norm{Z^y_{\cdot}(\zeta,\mu_0^y)-Z^y_{\cdot}(\zeta,\tilde{\mu}_0)}_{C([0,s])}\Bigr)\dd{\tilde{\mu}_0(\zeta,y)}\dd{s}.
\end{multline*}
By the monotonicity of the time integral, this yields in particular
\begin{multline*}
 \norm{Z_\cdot^{x}(\zeta^{\text{in}},\mu_0^x)-Z_\cdot^{x}(\zeta^{\text{in}},\tilde{\mu}_0)}_{C([0,t])}\\*
 \leq \sqrt{2}\max\{1,L_\W\}\ee^{L_K t}\dist(\mu_0,\tilde{\mu}_0)+L_K\int_{0}^{t}\norm{Z^x_\cdot(\zeta^{\text{in}},\mu_0)-Z^x_\cdot(\zeta^{\text{in}},\tilde{\mu}_0)}_{C([0,s])}\dd{s}\\*
 +L_K\int_{0}^{t}\int_{\R^d\times I}\norm{Z^y_{\cdot}(\zeta,\mu_0)-Z^y_{\cdot}(\zeta,\tilde{\mu}_0)}_{C([0,s])}\dd{\tilde{\mu}_0(\zeta,y)}\dd{s}.
\end{multline*}
Integrating with respect to $\dd{\tilde{\mu}_0}$ yields
\begin{multline*}
 \int_{\R\times I}\norm{Z_\cdot^{x}(\zeta^{\text{in}},\mu_0)-Z_\cdot^{x}(\zeta^{\text{in}},\tilde{\mu}_0)}_{C([0,t])}\dd{\tilde{\mu}_0^x(\zeta^{\text{in}},x)}\\*
 \shoveleft{\leq \sqrt{2}\max\{1,L_\W\}\ee^{L_K t} \dist(\mu_0,\tilde{\mu}_0)}\\*
 +2L_K\int_{0}^{t}\int_{\R^d\times I}\norm{Z^y_{\cdot}(\zeta,\mu_0)-Z^y_{\cdot}(\zeta,\tilde{\mu}_0)}_{C([0,s])}\dd{\tilde{\mu}_0(\zeta,y)}\dd{s}.
\end{multline*}
The claim then follows from Gronwall's inequality.
\end{proof}

We next derive Dobrushin's estimate, i.e.\@ a statement which provides stability with respect to the initial measure on finite time scales.

\begin{proposition}[Dobrushin's estimate]\label{Prop:Dobrushin}
 For $\mu_0,\tilde{\mu}_0\in \mathcal{P}_{1}(\R^d\times I)$ let $\mu_{t}=(Z_{t}^{\cdot}(\cdot,\mu_0),\id)\#\mu_0$ and $\tilde{\mu}_{t}=(Z_{t}^{\cdot}(\cdot,\tilde{\mu}_0),\id)\#\tilde{\mu}_0$ with $Z$ be given by Proposition~\ref{Prop:ex:charact:flow}. Then 
 \begin{equation*}
   \dist(\mu_t,\tilde{\mu}_{t})\leq \sqrt{2}\max\{1,L_\W\} \dist(\mu_0,\tilde{\mu}_0) 2\ee^{2L_K t}.
 \end{equation*}
\end{proposition}

\begin{proof}
 By definition, we have
 \begin{equation*}
  \dist(\mu_t,\tilde{\mu}_{t})=\sup_{\norm{\varphi}_{\text{Lip}}\leq 1}\abs[\bigg]{\int_{\R^d\times I}\varphi \dd{\mu}_t-\int_{\R^d\times I}\varphi \dd{\tilde{\mu}}_t}
 \end{equation*}
By definition of $\mu$ and $\tilde{\mu}$ we have together with Lemma~\ref{Lem:Lip:flow} that
\begin{multline*}
   \dist(\mu_t,\tilde{\mu}_{t})=\sup_{\norm{\varphi}_{\text{Lip}}\leq 1}\abs[\bigg]{\int_{\R^d\times I}\varphi(Z_{t}^{x}(\zeta,\mu_0),x) \dd{\mu_0(\zeta,x)}-\int_{\R^d\times I}\varphi(Z_{t}^{x}(\zeta,\tilde{\mu}_0),x) \dd{\tilde{\mu}_0(\zeta,x)}}\\*
   \leq \sup_{\norm{\varphi}_{\text{Lip}}\leq 1}\abs[\bigg]{\int_{\R^d\times I}\varphi(Z_{t}^{x}(\zeta,\mu_0),x) \dd{\mu_0(\zeta,x)}-\int_{\R^d\times I}\varphi(Z_{t}^{x}(\zeta,\mu_0),x) \dd{\tilde{\mu}_0(\zeta,x)}}\\*
   +\sup_{\norm{\varphi}_{\text{Lip}}\leq 1}\abs[\bigg]{\int_{\R^d\times I}\varphi(Z_{t}^{x}(\zeta,\mu_0),x)-\varphi(Z_{t}^{x}(\zeta,\tilde{\mu}_0),x) \dd{\tilde{\mu}_0(\zeta,x)}}\\*
   \leq \sqrt{2}\max\{1,L_\W\}\ee^{L_K t}\dist(\mu_0,\tilde{\mu}_0)+\int_{\R^d\times I}\abs{Z_{t}^{x}(\zeta,\mu_0)-Z_{t}^{x}(\zeta,\tilde{\mu}_0)} \dd{\tilde{\mu}_0(\zeta,x)}.
\end{multline*}
By means of Proposition~\ref{Prop:cont:dep:flow} we deduce
\begin{equation*}
   \dist(\mu_t,\tilde{\mu}_{t})   \leq \sqrt{2}\max\{1,L_\W\} \dist(\mu_0,\tilde{\mu}_0) 2\ee^{2L_K t}.
\end{equation*}
\end{proof}

\subsection{The mean-field limit}

In this section, we prove our main statement, i.e.\@ that the empirical measure corresponding to the discrete system \eqref{eq:disc:model:3} converges towards a corresponding continuous mean-field approximation in the limit of infinite particles $N\to\infty$. 
 We recall that the empirical measure corresponding to the solution $(\phi^k)$ of \eqref{eq:disc:model:4} is given as
\begin{equation*}
 \mu^{N}=\frac{1}{N}\sum_{k=1}^{N}\delta_{(\phi^{N,k}_\cdot,x^{k,N})}
\end{equation*}
 where $\delta$ denotes the Dirac delta measure. In analogy to classical mean-field systems, the solution $(\phi_t^k)$ to \eqref{eq:disc:model:3} is determined by the characteristic flow with respect to the empirical measure $\mu^{N}$. Precisely, we have the following statement.
 \begin{proposition}\label{Prop:repr:discr:sol}
  The solution $(\phi_t^{N,k})$ to \eqref{eq:disc:model:3} with initial data $(\phi^{N,k,\text{in}})$ is given by
  \begin{equation*}
   \phi_t^{N,k}=Z_t^{x^{k,N}}(\phi^{N,k,\text{in}},\mu^{N}_0)
  \end{equation*}
 \end{proposition}
 The proof of this statement is an immediate consequence of the uniqueness for the discrete system \eqref{eq:disc:model:4} and the equation for characteristics \eqref{eq:charact:2}.

We can now state and prove our main statement:

\begin{theorem}\label{Thm:mean:field:1}
 Assume that $\W\colon I\times I\to \R$ satisfies \eqref{eq:ass:Lip:W} and let $\phi^{N,k}$ be the solution to \eqref{eq:disc:model:3} with initial datum $(\phi^{N,k,\text{in}})$ and $w_{k\ell}^{\text{in}}$ given by \eqref{eq:Lip:weight}. Let $\mu^{N}=\frac{1}{N}\sum_{k=1}^{N}\delta_{(\phi^{N,k}_\cdot,x^{N,k})}$ be the corresponding empirical measure with $x^{N,k}$ satisfying \eqref{eq:disc}. Assume that $\mu_0^N$ satisfies $\dist(\mu_0^N,\mu_0)\to 0$ as $N\to\infty$ for some $\mu_0\in \mathcal{P}_1(\R^d\times I)$. Then $\mu^N$ converges to the mean-field limit given by $\mu=(Z_\cdot^\cdot(\cdot,\mu_0),\id)\#\mu_0$ where $Z$ is defined through \eqref{eq:charact:2}, i.e.\@ $\dist(\mu^N,\mu)\to 0$ as $N\to\infty$. Moreover, $\mu$ satisfies the following weak equation:
 \begin{multline}\label{eq:mfl}
  \frac{\dd}{\dd{t}}\int_{\R^d\times I}\varphi(\zeta,x)\dd{(e_t,\id)\#\mu(\zeta,x)}\\*
  =\int_{(C([0,T])\times I)\times (C([0,T])\times I)}\del_{\zeta}\varphi(e_t(\gamma),x)e_tK_{\cdot}(\W(x,y),r_t\gamma,r_t\tilde{\gamma})\dd{\mu(\tilde{\gamma},y)}\dd{\mu(\gamma,x)}
 \end{multline}
\end{theorem}

\begin{proof}
 The  convergence of $\mu^N$ to $\mu$ is an immediate consequence of Proposition~\ref{Prop:Dobrushin}. The second part follows from the definition of the characteristics.
\end{proof}

\begin{remark}
 The mean field limit $\mu$ is also characterised as the unique fixed-point via \eqref{eq:charact:1}. More precisely, for a given measure $\mu\in\mathcal{P}(C([0,T],\R^d)\times I)$ one can prove that there exists a unique solution to 
 \begin{equation}\label{eq:charact:3}
 \begin{split}
  \del_{t}Y^x_t(\zeta^{\text{in}},\mu)&=\int_{C([0,t])\times I}K_t(\W(x,y),r_t Y^x_\cdot(\zeta^{\text{in}},\mu),r_t \gamma_\cdot)\dd{\mu(\gamma,y)}=\vcc \mathcal{K}_t[\W,\mu](Y^x_\cdot(\zeta^{\text{in}},\mu),x,t)\\
  Y^x_0(\zeta^{\text{in}},\mu)&=\zeta^{\text{in}}.
 \end{split}
\end{equation}
For a given $\mu_0\in\mathcal{P}_1(\R^d\times I)$ one can then check $\mu$ from Theorem~\ref{Thm:mean:field:1} is the unique fixed point of the equation $\mu=(Y_\cdot^\cdot(\cdot,\mu),\id)\#\mu_0$. The procedure follows essentially the same lines as for static networks in e.g.\@ \cite{ChM16,KaM18}.
\end{remark}

\subsection{Relation to the continuum limit}\label{Sec:cont:lim}

In \cite{AyP21} a specific form of adaptive network model motivated by opinion formation has been considered, where each particle $z_k$ carries a dynamic weight $m_k$ whose evolution is coupled to the dynamics of the particles, i.e.\@ one has a system of the form
\begin{equation}\label{eq:opinion}
 \begin{split}
  \dot{z}_{k}&=\frac{1}{N}\sum_{\ell=1}^{N}m_\ell(t)g(z_\ell(t)-z_k(t))\\
  \dot{m}_k&=f_k(z,m)
 \end{split}
\end{equation}
where $z=(z_1,\ldots,z_N)$ and $m=(m_1,\ldots,m_N)$, i.e.\@ the weight dynamics depends on the whole set of particles and weights. For the specific form \eqref{eq:opinion}, including the particles' weights into the definition of the empirical measure, i.e.\@ $\mu^N=\sum_{k=1}^{N}m_k\delta_{z_k}$, it is possible to derive a classical mean-field equation up to an additional source term. Moreover, \cite{AyP21} noted a close connection between this mean-field equation and the continuum limit equation corresponding to \eqref{eq:opinion}. In the following we will briefly outline that there is an analogous relation for \eqref{eq:disc:model:3} and the mean-field limit \eqref{eq:mfl}. We note that in \cite{Thr24} the \emph{continuum limit} has been derived for adaptive network models including particularly \eqref{eq:disc:model:1}. However, the result can be generalised for non-local systems like \eqref{eq:disc:model:3} in a straightforward manner. More precisely, following the procedure in \cite{Thr24}, we define 
\begin{equation*}
 u^{N}(t,x)=\sum_{k=1}^{N}\phi_t^k \chi_{I_{k}}^{N}(x)
\end{equation*}
where $\phi_t^k$ is the solution to \eqref{eq:disc:model:3}. Moreover, we define the sequence of graphons
\begin{equation*}
 \W^N(x,y)=\sum_{k,\ell=1}^{N}\W(x_k,x_\ell)\chi_{I_k^N}(x)\chi_{I_\ell^N}(y).
\end{equation*}
Proceeding as in \cite{Thr24} (see also \cite{Med14,KaM17,Med19a}) one can show that, as $N\to\infty$, the step function $u^N$ converges to a limit $u$ which solves the equation
\begin{equation}\label{eq:cont:limit}
\del_{t}u(t,x)=\int_{I}K_t(\W(x,y),r_t u(\cdot,x),r_t u(\cdot,y))\dd{y}; \qquad u(0,\cdot)=u_0
\end{equation}
provided that $u^N(0,\cdot)$ converges to $u_0$ with respect to $\norm{\cdot}_{L^2}$.
\begin{remark}
 We note that \cite[Theorem 1.1]{Thr24} immediately applies to \eqref{eq:disc:model:1} which gives the continuum limit 
 \begin{equation*}
  \begin{split}
   \del_{t}u(t,x)&=\int_{I}w(t,x,y)C(u(t,x),u(t,y))\dd{y}\\
   \del_{t}w(t,x,y)&=F(w(t,x,y),u(t,x),u(t,y)).
  \end{split}
 \end{equation*}
 From this the equation \eqref{eq:cont:limit} follows for \eqref{eq:disc:model:1} using \eqref{eq:network:flow}.
\end{remark}
\begin{proposition}
Assume that $K$ satisfies \eqref{Ass:K1}--\eqref{Ass:K3} and $\W$ satisfies \eqref{eq:ass:Lip:W}. Let $u$ solve \eqref{eq:cont:limit}, then 
 \begin{equation*}
  \mu=\delta_{u(\cdot,x)}\lambda(x)
 \end{equation*}
with $\lambda$ denoting the Lebesgue measure on $I$, is a solution to the mean-field limit.
\end{proposition}
\begin{proof}
 By definition
 \begin{equation*}
  \del_{t}u(t,x)=\int_{I}K_t(\W(x,y),r_t u(\cdot,x),r_t u(\cdot,y))\dd{y}; \qquad u(0,\cdot)=u_0.
\end{equation*}
On the other hand from \eqref{eq:charact:2} with $\mu_0=e_0\#\mu$ and $\zeta^{\text{in}}=u_0$ we get
\begin{equation*}
  \del_{t}Z^x_t(u_0,\mu_0)=\int_{I}K_t(\W(x,y),r_t Z^x_\cdot(u_0,\mu_0),r_t Z^y_{\cdot}(u_0,\mu_0))\dd{y}.
\end{equation*}
By uniqueness of the solution we get $u(t,x)=Z_t^x(u_0,\mu_0)$ and thus 
\begin{equation*}
 (Z_\cdot^{\cdot}(\cdot,\mu_0),\id)\#\mu_0=\delta_{Z_\cdot^x(u_0,\mu_0)}\lambda(x)=\delta_{u(\cdot,x)}\lambda(x).
\end{equation*}
\end{proof}

\section{The mean-field limit for non-Lipschitz graphs}\label{Sec:non:Lip:graph}

In this section, we outline how the approximation of \eqref{eq:disc:model:3} can be generalised to initial networks which are generated from graphons without a Lipschitz condition. We repeat essentially the main steps from the previous section following mainly the notation and approach in \cite{Gol16} with suitable adaptations to the underlying initial network structure for which we rely on the approach in \cite{KaM18}. In particular, we assume that $\W$ satisfies the following continuity property
\begin{equation}\label{eq:W:cont}
 \int_{I}\abs{\W(x+\delta,y)-\W(x,y)}\dd{y}\to 0 \qquad \text{as } \delta\to 0 \quad \text{for a.e. } x\in I.
\end{equation}

\subsection{The characteristic equation}

We consider the following equation of characteristics with initial measure $\mu_0^x$.
\begin{equation}\label{eq:charact:1:non:Lip}
 \begin{split}
  \del_{t}Z^x_t(\zeta^{\text{in}},\W,\mu^x_0)&=\int_{I}\int_{C([0,t])}K_t(\W(x,y),r_t Z^x_\cdot(\zeta^{\text{in}},\W,\mu^x_0),\gamma_\cdot)\dd{\mu^{y}(\gamma)}\dd{y}
  \\
  \mu^{x}&=Z^x_{\cdot}(\cdot,\W,\mu^x_0)\#\mu^x_0\\
  Z^x_0(\zeta^{\text{in}},\W,\mu^x_0)&=\zeta^{\text{in}}
 \end{split}
\end{equation}
where $\zeta^{\text{in}}\in\R^d$, $\mu^x_0\in \mathcal{P}_1(\R^d)$. We consider $Z^x_{\cdot}(\cdot,\W,\mu^x_0)\colon \R^d\to C([0,T],\R^d)$ via $\R^d\ni \zeta\mapsto Z^x_{\cdot}(\zeta,\W,\mu^x_0)$

In analogy to \eqref{eq:charact:2}, we can reformulate \eqref{eq:charact:1} equivalently as
\begin{equation}\label{eq:charact:2:non:Lip}
 \begin{split}
  \del_{t}Z^x_t(\zeta^{\text{in}},\W,\mu^x_0)&=\int_{I}\int_{\R^d}K_t(\W(x,y),r_t Z^x_\cdot(\zeta^{\text{in}},\W,\mu^x_0),r_t Z^y_{\cdot}(\zeta,\W,\mu^y_0))\dd{\mu^y_0(\zeta)}\dd{y}
  \\*
  Z^x_0(\zeta^{\text{in}},\W,\mu^x_0)&=\zeta^{\text{in}}
 \end{split}
\end{equation}
We have well-posedness of the characteristic equation.
\begin{proposition}\label{Prop:ex:charact:flow:non:Lip}
 Assume that $K$ satisfies \eqref{Ass:K1}--\eqref{Ass:K3} and $W$ satisfies \eqref{eq:W:cont}. For each $\mu^x_0\in\mathcal{P}_1(\R^d)$ the equation \eqref{eq:charact:2:non:Lip} has a unique solution $[0,T]\times I\times \R^d \ni (t,x,\zeta^{\text{in}})\mapsto Z_{t}^{x}(\zeta^{\text{in}},\W,\mu^x_0)\in\R^d$. Moreover, $Z$ is differentiable with respect to the time variable $t$.
\end{proposition}
The proof is again a variation of the corresponding proof in \cite[Theorem 1.3.2]{Gol16} using a variation of Banach's fixed point argument and we omit the details.

The characteristics are still Lipschitz continuous with respect to the initial conditions. However, due to the lack of a Lipschitz condition for $\W$, we no longer have Lipschitz continuity with respect to $x$. More precisely, we have the following statement.
\begin{lemma}\label{Lem:Lip:flow:non:Lip}
 Assume that $K$ satisfies \eqref{Ass:K1}--\eqref{Ass:K3} and let $\W$ satisfy \eqref{eq:ass:Lip:W}. The characteristic flow $Z^x_t(\zeta^{\text{in}},\W,\mu^x_0)$ given by Proposition~\ref{Prop:ex:charact:flow:non:Lip} is Lipschitz continuous with respect to the initial datum $\zeta^{\text{in}}$. More precisely, 
   \begin{equation*}
  \norm{Z^x_\cdot(\zeta^{\text{in}},\W,\mu^x_0)-Z^x_\cdot(\tilde{\zeta}^{\text{in}},\W,\mu^x_0)}_{C([0,t])}\leq \ee^{L_K t}\abs{\zeta^{\text{in}}-\tilde{\zeta^{\text{in}}}}
 \end{equation*}
  uniformly in $\mu_0^x$ and $x$. In particular, we have the estimate
 \begin{equation*}
  \norm{Z^x_\cdot(\zeta^{\text{in}},\W,\mu^x_0)-Z^x_\cdot(\tilde{\zeta}^{\text{in}},\W,\mu^x_0)}_{C([0,T])}\leq \ee^{L_K T}\abs{\zeta^{\text{in}}-\tilde{\zeta^{\text{in}}}}.
 \end{equation*}
\end{lemma}

\begin{proof}
From \eqref{eq:charact:2:non:Lip} together with \eqref{Ass:K3} we deduce
 \begin{multline*}
 Z_t^{x}(\zeta^{\text{in}},\W,\mu_0^x)-Z_t^{x}(\tilde{\zeta}^{\text{in}},\W,\mu_0^x)=\zeta^{\text{in}}-\tilde{\zeta}^{\text{in}}\\*
\shoveleft{+\int_{0}^{t}\int_{I}\int_{\R^d}\Bigl(K_s(\W(x,y),r_sZ^x_\cdot(\zeta^{\text{in}},\W,\mu_0^x),r_sZ^y_{\cdot}(\zeta,\W,\mu_0^y))}\\*
\shoveright{-K_s(\W(x,y),r_s Z^x_\cdot(\tilde{\zeta}^{\text{in}},\W,\mu_0^x),r_s Z^y_{\cdot}(\zeta,\W,\mu_0^y))\Bigl)\dd{\mu_0^y(\zeta)}\dd{y}\dd{s}}\\*
\leq \zeta^{\text{in}}-\tilde{\zeta}^{\text{in}}+L_K\int_{0}^{t}\int_{I}\int_{\R^d}\norm{Z^x_\cdot(\zeta^{\text{in}},\W,\mu_0^x)-Z^x_\cdot(\tilde{\zeta}^{\text{in}},\W,\mu_0^x)}_{C([0,s])}\dd{\mu_0^y(\zeta)}\dd{y}\dd{s}\\*
\leq \zeta^{\text{in}}-\tilde{\zeta}^{\text{in}}+L_K\int_{0}^{t}\norm{Z^x_\cdot(\zeta^{\text{in}},\W,\mu_0^x)-Z^x_\cdot(\tilde{\zeta}^{\text{in}},\W,\mu_0^x)}_{C([0,s])}\dd{s}.
\end{multline*}
The claim follows upon taking the supremum of the absolute value and applying Gronwall's inequality.
\end{proof}

The next statement provides Lipschitz continuity with respect to the initial measure and the graphon. The latter property will compensate for the lack of Lipschitz continuity of $\W$.
\begin{proposition}\label{Prop:cont:dep:flow:non:Lip}
Assume that $K$ satisfies \eqref{Ass:K1}--\eqref{Ass:K3} and that $\W$ and $\tilde{\W}$ satisfy \eqref{eq:W:cont}. The characteristics provided by Proposition~\ref{Prop:ex:charact:flow:non:Lip} satisfy
\begin{multline*}
 \int_{I}\int_{\R^d}\norm{Z_\cdot^{x}(\zeta^{\text{in}},\W,\mu_0^x)-Z_\cdot^{x}(\zeta^{\text{in}},\tilde{\W},\tilde{\mu}_0^x)}_{C([0,t])}\dd{\tilde{\mu}_0^x(\zeta^{\text{in}})}\dd{x}\\*
 \leq \dist^I(\mu_0,\tilde{\mu}_0)(2\ee^{2L_K t}-\ee^{L_K t})+\frac{1}{2}\norm{\W-\tilde{\W}}_{L^1(I^2)}(\ee^{2L_K t}-1).
\end{multline*}
\end{proposition}

\begin{proof}
 From the definition of $Z$ via \eqref{eq:charact:2:non:Lip} we have
\begin{multline*}
 Z_t^{x}(\zeta^{\text{in}},\W,\mu_0^x)-Z_t^{x}(\zeta^{\text{in}},\tilde{\W},\tilde{\mu}_0^x)\\*
 \shoveleft{=\int_{0}^{t}\int_{I}\int_{\R^d}K_s(\W(x,y),r_s Z^x_\cdot(\zeta^{\text{in}},\W,\mu_0^x),r_s Z^y_{\cdot}(\zeta,\W,\mu_0^y))\dd{\mu_0^y(\zeta)}\dd{y}\dd{s}}\\*
\shoveright{-\int_{0}^{t}\int_{I}\int_{\R^d}K_s(\tilde{\W}(x,y),r_s Z^x_\cdot(\zeta^{\text{in}},\tilde{\W},\tilde{\mu}_0^x,r_s Z^y_{\cdot}(\zeta,\tilde{\W},\tilde{\mu}_0^y))\dd{\tilde{\mu}_0^y(\zeta)}\dd{y}\dd{s}}\\*
\shoveleft{=\int_{0}^{t}\int_{I}\int_{\R^d}K_s(\W(x,y),r_s Z^x_\cdot(\zeta^{\text{in}},\W,\mu_0^x),r_s Z^y_{\cdot}(\zeta,\W,\mu_0^y))\dd{\mu_0^y(\zeta)}}\\*
\shoveright{-\int_{I}\int_{\R^d}K_s(\W(x,y),r_s Z^x_\cdot(\zeta^{\text{in}},\W,\mu_0^x),r_s Z^y_{\cdot}(\zeta,\W,\mu_0^y))\dd{\tilde{\mu}_0^y(\zeta)}\dd{y}\dd{s}}\\
\shoveleft{+\int_{0}^{t}\int_{I}\int_{\R^d}\Bigl(K_s(\W(x,y),r_s Z^x_\cdot(\zeta^{\text{in}},\W,\mu_0^x),r_s Z^y_{\cdot}(\zeta,\W,\mu_0^y))}\\*
-K_s(\tilde{\W}(x,y),r_s Z^x_\cdot(\zeta^{\text{in}},\tilde{\W},\tilde{\mu}_0^x),r_s Z^y_{\cdot}(\zeta,\tilde{\W},\tilde{\mu}_0^y))\Bigl)\dd{\tilde{\mu}_0^y(\zeta)}\dd{y}\dd{s}.
\end{multline*}
Together with \eqref{Ass:K3} and Lemma~\ref{Lem:Lip:flow:non:Lip} this yields
\begin{multline*}
 Z_t^{x}(\zeta^{\text{in}},\W,\mu_0^x)-Z_t^{x}(\zeta^{\text{in}},\tilde{\W},\tilde{\mu}_0^x)\\*
 \shoveleft{\leq \ee^{L_K t}\dist^I(\mu_0,\tilde{\mu}_0)}\\*
 +L_K\int_{0}^{t}\int_{I}\int_{\R^d}\Bigl(\abs{\W(x,y)-\tilde{\W}(x,y)}+\norm{Z^x_\cdot(\zeta^{\text{in}},\W,\mu_0^x)-Z^x_\cdot(\zeta^{\text{in}},\tilde{\W},\tilde{\mu}_0^x)}_{C([0,s])}\\*
 +\norm{Z^y_{\cdot}(\zeta,\W,\mu_0^y)-Z^y_{\cdot}(\zeta,\tilde{\W},\tilde{\mu}_0^y)}_{C([0,s])}\Bigr)\dd{\tilde{\mu}_0^y(\zeta)}\dd{y}\dd{s}.
\end{multline*}
The monotonicity of the time integral further implies
\begin{multline*}
 \norm{Z_\cdot^{x}(\zeta^{\text{in}},\W,\mu_0^x)-Z_\cdot^{x}(\zeta^{\text{in}},\tilde{\W},\tilde{\mu}_0^x)}_{C([0,t])}\\*
 \shoveleft{\leq \ee^{L_K t}\dist^I(\mu_0,\tilde{\mu}_0)+tL_K \int_{I}\abs{\W(x,y)-\tilde{\W}(x,y)}\dd{y}}\\*
 +L_K\int_{0}^{t}\norm{Z^x_\cdot(\zeta^{\text{in}},\W,\mu_0^x)-Z^x_\cdot(\zeta^{\text{in}},\tilde{\W},\tilde{\mu}_0^x)}_{C([0,s])}\dd{s}\\*
 +L_K\int_{0}^{t}\int_{I}\int_{\R^d}\norm{Z^y_{\cdot}(\zeta,\W,\mu_0^y)-Z^y_{\cdot}(\zeta,\tilde{\W},\tilde{\mu}_0^y)}_{C([0,s])}\dd{\tilde{\mu}_0^y(\zeta)}\dd{y}\dd{s}.
\end{multline*}
Integrating with respect to $\dd{\tilde{\mu}_0^x}\dd{x}$ yields
\begin{multline*}
 \int_{I}\int_{\R^d}\norm{Z_\cdot^{x}(\zeta^{\text{in}},\W,\mu_0^x)-Z_\cdot^{x}(\zeta^{\text{in}},\tilde{\W},\tilde{\mu}_0^x)}_{C([0,t])}\dd{\tilde{\mu}_0^x(\zeta^{\text{in}})}\dd{x}\\
 \shoveleft{\leq \ee^{L_K t}\dist^I(\mu_0,\tilde{\mu}_0)+tL_K \norm{\W-\tilde{\W}}_{L^1(I^2)}}\\*
 +2L_K\int_{0}^{t}\int_{I}\int_{\R^d}\norm{Z^y_{\cdot}(\zeta,\W,\mu_0^y)-Z^y_{\cdot}(\zeta,\tilde{\W},\tilde{\mu}_0^y)}_{C([0,s])}\dd{\tilde{\mu}_0^y(\zeta)}\dd{y}\dd{s}.
\end{multline*}
The claim then follows from Gronwall's inequality.
\end{proof}

\subsection{Dobrushin's estimate}

We are now prepared to state and proof the generalisation of Dobrushin's estimate in the case of non-Lipschitz graph sequences.

\begin{proposition}[Dobrushin's estimate]\label{Prop:Dobrushin:non:Lip}
 Assume that $K$ satisfies \eqref{Ass:K1}--\eqref{Ass:K3} and that $\W$ satisfies \eqref{eq:W:cont}. For $\mu^x_0,\tilde{\mu}^x_0\in \mathcal{P}_{1}(\R^d)$ let $\mu_{t}^{x}=Z_{t}^{x}(\cdot,\mu^x_0)\#\mu^x_0$ and $\tilde{\mu}_{t}^{x}=Z_{t}^{x}(\cdot,\tilde{\mu}^x_0)\#\tilde{\mu}^{\text{in}}$ with $Z$ be given by Proposition~\ref{Prop:ex:charact:flow:non:Lip}. Then 
 \begin{equation*}
   \dist^{I}(\mu_t^{\cdot},\tilde{\mu}_{t}^{\cdot})\leq 2\dist^I(\mu_0,\tilde{\mu}_0)\ee^{2L_K t}+\frac{1}{2}\norm{\W-\tilde{\W}}_{L^1(I^2)}(\ee^{2L_K t}-1).
 \end{equation*}
\end{proposition}

\begin{proof}
 By definition, we have
 \begin{equation*}
  \dist^{I}(\mu_t^{\cdot},\tilde{\mu}_{t}^{\cdot})=\int_{I}\sup_{\norm{\varphi}_{\text{Lip}}\leq 1}\abs[\bigg]{\int_{\R}\varphi \dd{\mu}_t^x-\int_{\R}\varphi \dd{\tilde{\mu}}_t^x}\dd{x}
 \end{equation*}
By definition of $\mu$ and $\tilde{\mu}$ we have together with Lemma~\ref{Lem:Lip:flow:non:Lip} that
\begin{multline*}
   \dist^{I}(\mu_t^{\cdot},\tilde{\mu}_{t}^{\cdot})=\int_{I}\sup_{\norm{\varphi}_{\text{Lip}}\leq 1}\abs[\bigg]{\int_{\R^d}\varphi(Z_{t}^{x}(\zeta,\W,\mu^x_0)) \dd{\mu_0^x(\zeta)}-\int_{\R^d}\varphi(Z_{t}^{x}(\zeta,\tilde{\W},\tilde{\mu}^x_0)) \dd{\tilde{\mu}_0^x(\zeta)}}\dd{x}\\*
   \leq \int_{I}\sup_{\norm{\varphi}_{\text{Lip}}\leq 1}\abs[\bigg]{\int_{\R^d}\varphi(Z_{t}^{x}(\zeta,\W,\mu^x_0)) \dd{\mu_0^x(\zeta)}-\int_{\R^d}\varphi(Z_{t}^{x}(\zeta,\W,\mu^x_0)) \dd{\tilde{\mu}_0^x(\zeta)}}\dd{x}\\*
   +\int_{I}\sup_{\norm{\varphi}_{\text{Lip}}\leq 1}\abs[\bigg]{\int_{\R^d}\varphi(Z_{t}^{x}(\zeta,\W,\mu^x_0))-\varphi(Z_{t}^{x}(\zeta,\tilde{\W},\tilde{\mu}^x_0)) \dd{\tilde{\mu}_0^x(\zeta)}}\dd{x}\\*
   \leq \ee^{2L_K T}\dist^I(\mu_0,\tilde{\mu}_0)+\int_{I}\int_{\R^d}\abs{Z_{t}^{x}(\zeta,\W,\mu^x_0)-Z_{t}^{x}(\zeta,\tilde{\W},\tilde{\mu}^x_0)} \dd{\tilde{\mu}_0^x(\zeta)}\dd{x}.
\end{multline*}
By means of Proposition~\ref{Prop:cont:dep:flow:non:Lip} we deduce
\begin{equation*}
   \dist^{I}(\mu_t^{\cdot},\tilde{\mu}_{t}^{\cdot})\leq 2\dist^I(\mu_0,\tilde{\mu}_0)\ee^{2L_K t}+\frac{1}{2}\norm{\W-\tilde{\W}}_{L^1(I^2)}(\ee^{2L_K t}-1).
\end{equation*}
\end{proof}

\subsection{The mean-field limit and particle approximation}

In this section, we provide the mean-field limit of \eqref{eq:disc:model:3} for non-Lipschitz graph sequences. Following the construction in \cite{KaM18}, the latter is approximated by the empirical measure of the discrete system \eqref{eq:disc:model:3} which is a consequence of Dobrushin's estimate and suitable intermediate approximations. 

More precisely, for a given $\W\colon I\times I \to \R$ satisfying \eqref{eq:W:cont}, we define weights according to \eqref{eq:average:weight}, i.e.\@
 \begin{equation*}
  w_{k\ell}^{N}=N^2\int_{I^N_k\times I^N_\ell}\W(x,y)\dd{x}\dd{y}.
 \end{equation*}
The sequence of weighted graphs $G^N=(V(G^N),E(G^N))$ is then given by the set of vertices $V(G^N)=[N]$ and edges $E(G^N)=\{(k,\ell)\in I\times I \;\vert\; w^N_{k\ell}\neq 0\}$. The corresponding sequence of (step) graphons is given by
\begin{equation}\label{eq:graph:seq:average}
 \W^N(x,y)=\sum_{k,\ell=1}^{N}w^N_{k\ell}\chi_{I_k^N}(x)\chi_{I_\ell^N}(y).
\end{equation}
We note that \cite[Lemma 3.3]{KaM18} ensures that $\W^N\to \W$ in $L^2(I)$ as $N\to\infty$.

 Following the approach in \cite{KaM18} one can construct a particle system satisfying \eqref{eq:disc:model:3} which approximates the mean-field limit. More precisely, for $n,m\in\N$ let $N=mn$ and consider the system
\begin{equation}\label{eq:approx:syst}
  \dot{\phi}^{N,(k-1)m+\ell}_t=\frac{1}{N}\sum_{i=1}^{n}\sum_{j=1}^{m}K_{t}(w_{k\ell}^{N},\phi|^{N,(k-1)m+\ell}_{[0,t]},\phi^{N,(i-1)m+j}|_{[0,t]});\qquad \phi^{N,(k-1)m+\ell}_0=\phi^{N,k \ell,\text{in}}
\end{equation}
for $k\in[n]$ and $\ell\in[m]$  with initial data $\phi^{N,k \ell,\text{in}}$ such that for fixed $k\in[n]$ the sequence $(\phi^{N,k \ell,\text{in}})_{\ell\in\N}$ consists of independent identically distributed random variables with distribution 
\begin{equation*}
 \mu_0^{n,k}=n\int_{I^n_k}\mu_0^{x}\dd{x}.
\end{equation*}
This defines a probability measure $\mathbb{P}_0^{n}$ on the measure space $\Omega^{n}=(((\R^d)^\infty)^n,\mathcal{B}((\R^d)^{\infty})^n)$

Consider the corresponding local empirical measure which is given by
\begin{equation}\label{eq:emp:meas:non:Lip}
 \mu^{n,m,x}_t=\sum_{i=1}^{n}\chi_{I^n_i}(x)\frac{1}{m}\sum_{j=1}^{m}\delta_{\phi^{N,(k-1)m+\ell}_t}.
\end{equation}
Let moreover
\begin{equation}\label{eq:lim:meas}
 \mu_t^{x}=Z_t^x(\cdot,\W,\mu_0^x)\#\mu_0^x
\end{equation}
with the characteristic $Z$ given by Proposition~\ref{Prop:ex:charact:flow:non:Lip}. Analogously to \cite{KaM18} we have the following approximation result:
\begin{theorem}\label{Thm:mean:field:2}
 Let $\mu_0^x\in \mathcal{P}^I_1(\R^d)$ be given and assume that $\W\colon I\times I\to \R$ satisfies \eqref{eq:W:cont}, $K$ satisfies \eqref{Ass:K1}--\eqref{Ass:K3} and let $w_{k\ell}^{\text{in}}$ be given by \eqref{eq:average:weight}. Let $\phi^{N,k}$ solve \eqref{eq:approx:syst} with initial datum $(\phi^{N,k,\text{in}})$. Let $ \mu^{n,m,x}_t$ be the corresponding empirical measure given by \eqref{eq:emp:meas:non:Lip}. Then $\mu^{n,m,\cdot}$ converges to the mean-field limit $\mu_t^x$ given by \eqref{eq:lim:meas},
 i.e.\@ $\dist^I(\mu^{n,m,\cdot},\mu^{\cdot})\to 0$ as $n,m\to\infty$. Moreover, $\mu$ satisfies the following weak equation:
 \begin{multline}\label{eq:mfl:2}
  \frac{\dd}{\dd{t}}\int_{I}\int_{\R^d}\varphi(\zeta)\dd{e_t\#\mu^x(\zeta)}\dd{x}\\*
  =\int_{I}\int_{C([0,T])}\int_{I}\int_{C([0,T])}\varphi'\circ e_t(\gamma)(e_tK_{\cdot})(\W(x,y),r_t\gamma,r_t\tilde{\gamma})\dd{\mu^y(\tilde{\gamma})}\dd{y}\dd{\mu^x(\gamma)}\dd{x}
 \end{multline}
\end{theorem}

The proof of this statement is a consequence of suitable intermediate approximations in combination with Dobrushin's estimate (Proposition~\ref{Prop:Dobrushin:non:Lip}) and follows the same lines as the corresponding proof of \cite[Theorem~3.11]{KaM18} to where we refer for details. The same construction has also been used e.g.\@ in \cite{KuX22a}.

\section*{Acknowledgements}

The author thanks Rishabh Gvalani for a valuable discussion.

\end{document}